\documentclass[10pt,a4paper,reqno,sumlimits]{amsart}
\usepackage[numbers]{natbib}
\usepackage{amsmath}
\usepackage{amssymb}
\usepackage[mathscr]{euscript}
\usepackage{enumerate}
\usepackage{xspace}
\usepackage{color}
\usepackage{enumerate}
\usepackage{enumitem}
\usepackage{amsthm}

\begin{document}

\theoremstyle{plain}
\newtheorem{C}{Convention}
\newtheorem*{SA}{Standing Assumption}
\newtheorem{theorem}{Theorem}[section]
\newtheorem{condition}{Condition}[section]
\newtheorem{lemma}[theorem]{Lemma}
\newtheorem{proposition}[theorem]{Proposition}
\newtheorem{corollary}[theorem]{Corollary}
\newtheorem{claim}[theorem]{Claim}
\newtheorem{definition}[theorem]{Definition}
\newtheorem{Ass}[theorem]{Assumption}
\newcommand{\q}{Q}
\theoremstyle{definition}
\newtheorem{remark}[theorem]{Remark}
\newtheorem{note}[theorem]{Note}
\newtheorem{example}[theorem]{Example}
\newtheorem{assumption}[theorem]{Assumption}
\newtheorem*{notation}{Notation}
\newtheorem*{assuL}{Assumption ($\mathbb{L}$)}
\newtheorem*{assuAC}{Assumption ($\mathbb{AC}$)}
\newtheorem*{assuEM}{Assumption ($\mathbb{EM}$)}
\newtheorem*{assuES}{Assumption ($\mathbb{ES}$)}
\newtheorem*{assuM}{Assumption ($\mathbb{M}$)}
\newtheorem*{assuMM}{Assumption ($\mathbb{M}'$)}
\newtheorem*{assuL1}{Assumption ($\mathbb{L}1$)}
\newtheorem*{assuL2}{Assumption ($\mathbb{L}2$)}
\newtheorem*{assuL3}{Assumption ($\mathbb{L}3$)}
\newtheorem{charact}[theorem]{Characterization}
\newcommand{\notiz}{\textup} 
\renewenvironment{proof}{{\parindent 0pt \it{ Proof:}}}{\mbox{}\hfill\mbox{$\Box\hspace{-0.5mm}$}\vskip 16pt}
\newenvironment{proofthm}[1]{{\parindent 0pt \it Proof of Theorem #1:}}{\mbox{}\hfill\mbox{$\Box\hspace{-0.5mm}$}\vskip 16pt}
\newenvironment{prooflemma}[1]{{\parindent 0pt \it Proof of Lemma #1:}}{\mbox{}\hfill\mbox{$\Box\hspace{-0.5mm}$}\vskip 16pt}
\newenvironment{proofcor}[1]{{\parindent 0pt \it Proof of Corollary #1:}}{\mbox{}\hfill\mbox{$\Box\hspace{-0.5mm}$}\vskip 16pt}
\newenvironment{proofprop}[1]{{\parindent 0pt \it Proof of Proposition #1:}}{\mbox{}\hfill\mbox{$\Box\hspace{-0.5mm}$}\vskip 16pt}
\newcommand{\s}{\u s}

\newcommand{\Law}{\ensuremath{\mathop{\mathrm{Law}}}}
\newcommand{\loc}{{\mathrm{loc}}}
\newcommand{\Log}{\ensuremath{\mathop{\mathscr{L}\mathrm{og}}}}
\newcommand{\Meixner}{\ensuremath{\mathop{\mathrm{Meixner}}}}
\newcommand{\of}{[\hspace{-0.06cm}[}
\newcommand{\gs}{]\hspace{-0.06cm}]}

\let\MID\mid
\renewcommand{\mid}{|}

\let\SETMINUS\setminus
\renewcommand{\setminus}{\backslash}

\def\stackrelboth#1#2#3{\mathrel{\mathop{#2}\limits^{#1}_{#3}}}

\renewcommand{\theequation}{\thesection.\arabic{equation}}
\numberwithin{equation}{section}

\newcommand\llambda{{\mathchoice
      {\lambda\mkern-4.5mu{\raisebox{.4ex}{\scriptsize$\backslash$}}}
      {\lambda\mkern-4.83mu{\raisebox{.4ex}{\scriptsize$\backslash$}}}
      {\lambda\mkern-4.5mu{\raisebox{.2ex}{\footnotesize$\scriptscriptstyle\backslash$}}}
      {\lambda\mkern-5.0mu{\raisebox{.2ex}{\tiny$\scriptscriptstyle\backslash$}}}}}

\newcommand{\prozess}[1][L]{{\ensuremath{#1=(#1_t)_{0\le t\le T}}}\xspace}
\newcommand{\prazess}[1][L]{{\ensuremath{#1=(#1_t)_{0\le t\le T^*}}}\xspace}

\newcommand{\tr}{\operatorname{tr}}
\newcommand{\lijepoa}{{\mathscr{A}}}
\newcommand{\lijepob}{{\mathscr{B}}}
\newcommand{\lijepoc}{{\mathscr{C}}}
\newcommand{\lijepod}{{\mathscr{D}}}
\newcommand{\lijepoe}{{\mathscr{E}}}
\newcommand{\lijepof}{{\mathscr{F}}}
\newcommand{\lijepog}{{\mathscr{G}}}
\newcommand{\lijepok}{{\mathscr{K}}}
\newcommand{\lijepoo}{{\mathscr{O}}}
\newcommand{\lijepop}{{\mathscr{P}}}
\newcommand{\lijepoh}{{\mathscr{H}}}
\newcommand{\lijepom}{{\mathscr{M}}}
\newcommand{\lijepou}{{\mathscr{U}}}
\newcommand{\lijepov}{{\mathscr{V}}}
\newcommand{\lijepoy}{{\mathscr{Y}}}
\newcommand{\cF}{{\mathscr{F}}}
\newcommand{\cG}{{\mathscr{G}}}
\newcommand{\cH}{{\mathscr{H}}}
\newcommand{\cM}{{\mathscr{M}}}
\newcommand{\cD}{{\mathscr{D}}}
\newcommand{\bD}{{\mathbb{D}}}
\newcommand{\bF}{{\mathbb{F}}}
\newcommand{\bG}{{\mathbb{G}}}
\newcommand{\bH}{{\mathbb{H}}}
\newcommand{\dd}{\operatorname{d}\hspace{-0.05cm}}
\newcommand{\ddd}{\operatorname{d}}
\newcommand{\er}{{\mathbb{R}}}
\newcommand{\ce}{{\mathbb{C}}}
\newcommand{\erd}{{\mathbb{R}^{d}}}
\newcommand{\en}{{\mathbb{N}}}
\newcommand{\de}{{\mathrm{d}}}
\newcommand{\im}{{\mathrm{i}}}
\newcommand{\indik}{{\mathbf{1}}}
\newcommand{\D}{{\mathbb{D}}}
\newcommand{\E}{E}
\newcommand{\N}{{\mathbb{N}}}
\newcommand{\Q}{{\mathbb{Q}}}
\renewcommand{\P}{{\mathbb{P}}}
\newcommand{\ud}{\operatorname{d}\!}
\newcommand{\ii}{\operatorname{i}\kern -0.8pt}
\newcommand{\cadlag}{c\`adl\`ag }
\newcommand{\p}{P}
\newcommand{\F}{\mathbf{F}}
\newcommand{\1}{\mathbf{1}}
\newcommand{\f}{\mathscr{F}^{\hspace{0.03cm}0}}
\newcommand{\lle}{\langle\hspace{-0.085cm}\langle}
\newcommand{\rre}{\rangle\hspace{-0.085cm}\rangle}
\newcommand{\llbr}{[\hspace{-0.085cm}[}
\newcommand{\rrbr}{]\hspace{-0.085cm}]}

\def\EM{\ensuremath{(\mathbb{EM})}\xspace}

\newcommand{\la}{\langle}
\newcommand{\ra}{\rangle}

\newcommand{\Norml}[1]{%
{|}\kern-.25ex{|}\kern-.25ex{|}#1{|}\kern-.25ex{|}\kern-.25ex{|}}

\title[Real-World and Risk-Neutral Dynamics for HJM Frameworks]{A Note On Real-World and Risk-Neutral Dynamics for Heath-Jarrow-Morton Frameworks} 
\author[D. Criens]{David Criens}
\address{D. Criens - Technical University of Munich, Center for Mathematics, Germany}
\email{david.criens@tum.de}

\keywords{Heath-Jarrow-Morton Framework, L\'evy Term Structure Models, Real-World Dynamics, Risk-Neutral Dynamics, Stochastic Partial Differential Equations, Changes of Measures\vspace{1ex}}

\subjclass[2010]{91G30,	60J25, 60H15}

\thanks{D. Criens - Technical University of Munich, Center for Mathematics, Germany,  \texttt{david.criens@tum.de}.}

\date{\today}
\maketitle

\frenchspacing
\pagestyle{myheadings}

\begin{abstract}
As a consequence of the financial crises, risk management became more important and real-world dynamics of interest-rate models moved into the focus of interest. 
Since risk-neutral dynamics are classically important to compute prices of financial derivatives, it is interesting when real-world dynamics can be related to risk-neutral dynamics via an equivalent change of measures. 
In this article we give deterministic conditions in a general Heath-Jarrow-Morton framework driven by a Hilbert space valued Brownian motion and a Poisson random measure. Our conditions are of Lipschitz type and therefore easy to verify.
\end{abstract}

\section{Introduction}
%
%
Most of the existing literature on interest rate modeling is restricted to risk-neutral dynamics. 
Mainly, this is because in bond markets future values of securities are observable and the risk-neutral dynamics can be calibrate directly to the observed prices. Thus, if one is only interested in the computation of prices, it suffices to consider a risk-neutral model.
However, since the financial crises the importance of risk management has grown and it is a consensus among quants that real-world dynamics should be given more attention, see \cite{doi:10.1142/S0219024905003049, hull, rebonato2005evolving} for discussions.

To cover both applications, risk-neutral pricing and risk management, it is indispensable to model real-world and risk-neutral dynamics simultaneously, i.e., in mathematical terms, to related them via an equivalent change of measures.

In (finite dimensional) stock markets these equivalent changes of measures are well-studied and many precise conditions are known, see, e.g., \cite{C16SP, EJ97, KS(2002b), protter}. 
In the infinite-dimensional setup of interest rates the literature is not excessive.
The existence of an equivalent change of measures in a Heath-Jarrow-Morton setting driven by Hilbert space valued Brownian motion is verified in \cite{filipovic2001consistency} under an exponential moment condition. To the best of our current knowledge, this is the only general study in a Heath-Jarrow-Morton framework.

In this note we discuss a Markovian Heath-Jarrow-Morton model framework driven not only by a Hilbert space valued Brownian motion, but also a Poisson random measure. 
Broadly speaking, we show that if risk-neutral and real-world dynamics can be defined and related via a drift condition, then it suffices that one of the dynamics have a unique law to conclude the existence of an equivalent change of measures.
In particular, this yields that in the Markovian version of the setting in \cite{filipovic2001consistency} the imposed exponential moment condition can be dropped without replacement. 
We also give deterministic conditions which are easy to verify.
In the case of the risk-neutral dynamics, these are taken from \cite{FTT1}, where it is also noted that they cannot be weakened substantially, see also \cite{morton1989arbitrage}.
Based on other results in \cite{FTT1}, we also give deterministic conditions such that the real-world and the risk-neutral dynamics only produce non-negative forward curves.
%
%

Let us shortly comment on our methodology. We introduce our financial market in the spirit of Musiela's parameterization \cite{musiela} via cylindrical martingale problems.
We stress that, under mild assumptions on the coefficients, this formalisms is equivalent to the more classical formulation based on mild solutions to stochastic partial differential equations (SPDEs). 

The following section is structured as follows.
In Section \ref{sec: 2.1}, we introduce our mathematical framework. Then, in Section \ref{sec: 2.2}, we study the existence of real-world and risk-neutral dynamics and in Section \ref{sec: 2.3} we give conditions such that these dynamics only produce non-negative forward curves. 


\section{HJMM Pairs}
\subsection{The Space of Forward Curves and Cylindrical Martingale Problems}\label{sec: 2.1}
We introduce the space of forward curves following \cite{FTT1}.
For \(\beta \in (0, \infty)\), let \(H_\beta\) be the space of absolutely continuous functions \(f \colon [0, \infty) \to \mathbb{R}\) such that 
\[
\|h\|_\beta \equiv \left( |h(0)|^2 + \int_0^\infty |h'(z)|^2 e^{\beta z} \dd z \right)^{\frac{1}{2}} < \infty.
\]
It is well-known that \(H_\beta\) is a separable Hilbert space and the shift semigroup is strongly continuous with generator \(A \equiv \frac{\dd }{\dd z}\), see \cite[Theorem 2.1]{FTT1}. 
Moreover, we define \(H^o_\beta \equiv \{f \in H_\beta \colon \lim_{z \to \infty} f(z) = 0\}\), which is a closed subspace of \(H_\beta\). 

Fix a separable Hilbert space \(U\) and a standard Borel space \((E, \mathscr{E})\), i.e. a Hausdorff measurable space \(E\) such that \(\mathscr{E}\) is countably generated and \(\sigma\)-isomorphic to the Borel \(\sigma\)-field of a Polish space. One can always think of \((E, \mathscr{E})\) to be \((\mathbb{R}^d, \mathscr{B}(\mathbb{R}^d))\) or, more generally, a Polish space equipped with its Borel \(\sigma\)-field.
Now, let \(K\) be a trace-class operator on \(U\) and let \(q(\dd t, \dd x) = \dd t \otimes F(\dd x)\) be the intensity measure of a homogeneous Poisson random measure (see \cite[Definition II.1.20]{JS}) on \((E, \mathscr{E})\). In this case, \(F\) is a \(\sigma\)-finite measure on \((E, \mathscr{E})\).
The operator \(K\) will be the covariance of the driving Brownian motion. Assuming \(K\) to be of trace-class is not restrictive. To see this, recall that for a cylindrical Brownian motion on a separable Hilbert space there is always a bigger separable Hilbert space on which the cylindrical Brownian motion is of trace-class, see, for instance, \cite{DePrato, Mikulevicius1998}.
The space \(U^o \equiv K^{\frac{1}{2}} (U)\) with the scalar product \(\la x, y \ra_{U^o} \equiv \la K^{-\frac{1}{2}} x, K^{- \frac{1}{2}} y\ra_U\), where \(x, y \in U^o\), is a separable Hilbert space.
We denote by \(L^2\) the space of Hilbert-Schmidt operators \(U^o \to H_\beta^o\) and note that the corresponding Hilbert-Schmidt scalar product is given by 
\[
\la \Phi, \Psi\ra_{L^2} = \left\la \Phi K^{\frac{1}{2}}, \Psi K^{\frac{1}{2}}\right\ra_{\textup{HS}}, \qquad \Phi, \Psi \in L^2,
\] 
where \(\la \cdot, \cdot\ra_{\textup{HS}}\) is the Hilbert-Schmidt scalar product on the space of Hilbert-Schmidt operators \(U \to H_\beta\). 
Moreover, let \(\beta' > \beta\).

Next, we introduce the parameters of our real-world market. 
Let \(a \colon H_\beta \to L^2, b \colon H_\beta \to H^o_\beta\) and \(\gamma \colon H_\beta \times E \to H^o_{\beta'}\) be Borel functions and \(h_0 \in H_\beta\). 
We suppose that \(b, a\) and \(x \mapsto \int 1 \wedge \|\gamma(x, y)\|^2_{H_{\beta'}} F(\dd y)\) are bounded on bounded subsets of \(H_\beta\). 

%
Define \(\Omega\) to be the space of \cadlag functions \([0, T] \to H_\beta\) and \(\mathbf{F} \equiv (\mathscr{F}_t)_{t \in [0, T]}\) the canonical filtration generated by the coordinate process \(X\), i.e. \(\mathscr{F}_t = \sigma(X_s, s \in [0, t])\) and \(X_t(\omega) = \omega(t)\) for \(\omega \in \Omega\). In particular, we set \(\mathscr{F} \equiv \mathscr{F}_T\).
\begin{definition}
	We call a probability measure \(P\) on \((\Omega, \mathscr{F}, \mathbf{F})\) a solution to the martingale problem (MP) associated with \((F, b, a, \gamma, h_0)\), if for all \(n \in \mathbb{N}\), \(f \in C^2_c(\mathbb{R}^n)\) and \(y_1, ..., y_n\) in the domain of \(A^*\), the adjoint of \(A\), the process
	\[
	f^* (X_\cdot) - f^*(h_0) - \int_0^\cdot \mathcal{K} f^* (X_{s-})\dd s
	\]
	is a local \((\F, P)\)-martingale. Here, \(f^*(z) = f(\la z, y_1\ra_{H_\beta}, ..., \la z, y_n\ra_{H_\beta})\) and 
	\begin{align*}
	\mathcal{K} f^*(z) \equiv  \sum_{i = 1}^n (& \langle z, A^* y_i\rangle_{H_\beta} + \langle b(z), y_i \rangle_{H_\beta})\partial_i f^*(z) \\&+  \frac{1}{2} \sum_{i=1}^n\sum_{j = 1}^n \langle a(z) a^*(z)y_i, y_j\rangle_{H_\beta} \partial^2_{ij}f^*(z) 
	\\&+ \int \left(f(\gamma(z, y) + y) - f(z) - \sum_{i = 1}^n \langle \gamma(z, y), y_i\rangle_{H_\beta} \partial_i f^*(z) \right) F(\dd y),
	\end{align*}
	where \(a^*\) denotes the adjoint of \(a\), 
	\[
	\partial_i f^*(z) \equiv (\partial_i f)(\la z, y_1\ra_{H_\beta}, ..., \la z, y_n\ra_{H_\beta})
	\]
	and \(\partial^2_{ii} f^*\) is defined in a similar manner.
	We call a MP \((F, b, a, \gamma)\) to be \emph{well-posed}, if there exists a unique solution for all initial values \(h_0 \in H_\beta\).
\end{definition}
Under very mild conditions, the set of solutions to cylindrical martingale problems coincides with the set of laws of (mild and (analytically) weak) solution processes to the Heath-Jarrow-Morton-Musiela SPDE 
\[
\dd r_r = \left(A r_t + b(r_t)\right) \dd t + a(r_t)\dd W_t + \int_E\gamma (r_{t-}, x) \left(\mu(\dd t, \dd x) - p(\dd t, \dd x)\right), 
\]
where \(W\) is a Brownian motion and \(\mu\) is a Poisson random measure with intensity measure \(p\).
For formal statements we refer to \cite[Theorem 3.12]{criens17b} and \cite[Lemmata 7.6 and 7.7]{FTT2}.

\subsection{Existence of HJMM Pairs}\label{sec: 2.2}
Next, we define a pair of a real-world and a risk-neutral Heath-Jarrow-Morton-Musiela-type financial market.

\begin{definition}
	We call a tuple \((P, Q)\) of probability measures on \((\Omega, \mathscr{F}, \mathbf{F})\) a \emph{HJMM pair} associated with \((F, b,a, \gamma, h_0)\), if \(P\) solves the MP \((F, b, a, \gamma, h_0)\), \(Q \sim P\) and for all maturities \(T^* \in [0, T]\) the discounted bond price
	\begin{align}\label{eq: dbp}
	\exp \left( - \int_0^\cdot X_s(0) \dd s - \int_0^{T^*- \cdot} X_s(z)\dd z\right)
	\end{align}
	is a local \((\F,Q)\)-martingale (time indexed on \([0, T^*]\)). 
\end{definition}
In other words, if \((P, Q)\) is a HJMM pair, then under \(P\) the coordinate process has dynamics which correspond to a typical Heath-Jarrow-Morton-Musiela market driven by a Brownian motion and a Poisson random measure, and \(Q\) is an equivalent local martingale measure (EMM) for the corresponding market. Furthermore, to reduce the complexity of the model, it is preferable that \(Q\) is \emph{structure preserving} in the sense that it also solves some MP. 

We observe the following: Let \((P, Q)\) be a pair of solutions to MPs from which (at least) one is well-posed. If the discounted bond prices are local \(Q\)-martingale for all maturities, a drift condition suffices to conclude that \((P, Q)\) is a HJMM pair.
The crucial point behind this is that the drift condition together with the uniqueness assumption already implies that \(P\) and  \(Q\) are equivalent. In particular, no exponential moment condition is needed. 

We now formalize this observation.  For later reference we state the drift-condition separately. Let \((\lambda_j)_{j \in \mathbb{N}}\) be the eigenvalues of \(K\) and \((e_j)_{j \in \mathbb{N}}\) be the corresponding eigenvectors. Set \(a_j (h) \equiv \sqrt{\lambda_j} a(h) e_j\) for \(h \in H_\beta\) and \(j \in \mathbb{N}\).
\begin{condition}\label{cond: DC}
	There are Borel functions \(\zeta \colon H_\beta \to H_\beta\) and \(Y \colon H_\beta \times H_{\beta} \to (0, \infty)\) such that \(\la a a^* \beta, \beta\ra_{H_\beta}\) and
	\begin{align*}
	&h\mapsto \int \left(1 - \sqrt{Y(h, \gamma(h, z))}\right)^2 F(\dd z)
	\end{align*}
	are bounded on bounded subsets of \(H_\beta\). Moreover, there exist an intensity measure \(p'(\dd t, \dd z) = \dd t \otimes F'(\dd z)\) of a homogeneous Poisson random measure on \((E, \mathscr{E})\) and a Borel function \(\gamma' \colon H_\beta \times E \to H_{\beta'}^o\) such that
	\(x \mapsto \int 1 \wedge \|\gamma'(x, z)\|^2_{H_{\beta'}} F'(\dd z)\) and
	\begin{align*}
	\xi(x) \equiv  & \sum_{j \in \mathbb{N}} \left(a_j(x) \int_0^\cdot a_j(x)(z)\dd z\right) 
	\\&\qquad- \int \gamma' (x, z) \left( \exp \left(- \int_0^\cdot \gamma'(x, z) (y) \dd y\right) - 1\right) F'(\dd z)
	\end{align*}
	are bounded on bounded subsets of \(H_\beta\). In particular, \(\xi\colon H_\beta \to H_\beta^o\).
	Finally, the following equations hold for all \(x \in H_\beta\) and \(G \in \mathscr{E}\):
	\begin{align}\label{eq: MPRE}
	\xi(x)
	&= b(x) + a(x)a^*(x) \zeta(x) + \int \gamma(x, y) (Y(x, \gamma(x, y)) - 1) F(\dd y),				
	\end{align} 
	and 
	\begin{align}\label{eq: Y cond}
	\int \1_G (\gamma(x, y)) Y(x, \gamma (x, y)) F(\dd y) = \int \1_G (\gamma'(x, y)) F'(\dd y).
	\end{align}
\end{condition}
	The equation \eqref{eq: MPRE} is usually called \emph{market price of risk equation (MPRE)}. In the case of markets driven by Brownian motion, the equation initiates typically a unique EMM. In particular, in the continuous case the EMM only depends on the volatility coefficient \(a\). 
	We stress that infinitely many HJMM pairs may correspond to one specify volatility coefficient \(a\). 
	
The identity \eqref{eq: Y cond} can be used to influence, for instance, integrability properties of \(F\). More precisely, if we can take \(Y(x, y) \equiv Y(x)\) and \(\gamma' \equiv \gamma\), then \(F' (\dd x) = Y(x) F(\dd x)\).
In this case, \(Y\) can be thought as a weight, such that \(F'\) has, say, exponential moments, while \(F\) does not.
Typically this is possible if the volatility coefficient \(a\) is invertible, but there are also other scenarios. 

In classical discussions of real-wold HJMM markets driven by Brownian motion, it is assumed that
\[
b (x) = \sum_{j \in \mathbb{N}} \left(a_j(x) \int_0^\cdot a_j(x)(z)\dd z\right) - a(x) a^*(x) \zeta(x)
\]
and Condition \ref{cond: DC} holds, see, e.g., \cite{carmona2007interest, filipovic2001consistency}. 
\begin{theorem}\label{theo: A1}
	Assume that Condition \ref{cond: DC} holds and suppose that \(P\) is a solution to the MP \((F, b, a, \gamma, h_0)\) and that \(Q\) is a solution to the MP \((F', \xi, a, \gamma', h_0)\) such that for all maturities \(T^*\in [0, T]\) the discounted bond prices \eqref{eq: dbp} are local \((\F,Q)\)-martingales (time indexed on \([0, T^*]\)). Then \((P, Q)\) is a HJMM pair whenever (at least) one of the MPs \((F, b, a, \gamma)\) and \((F', \xi, a, \gamma')\) is well-posed.
\end{theorem}
\begin{proof}
The theorem follows from \cite[Proposition 3.9]{criens17b}.
Let us explain the key ideas behind \cite[Proposition 3.9]{criens17b} under the assumption that the MP \((F, b, a, \gamma)\) is well-posed. For details we refer to \cite{criens17b}.
First, one can find a cylindrical local \(P\)-martingale \(X^c\) such that for all \(k \in H_\beta\) the real-valued process \(\la X^c, k\ra_{H_\beta}\) has quadratic variation process \(\int_0^\cdot \la a(X_s) a^*(X_s)k , k \ra_{H_\beta} \dd s\). To obtain this result, one has to extend a family of local martingales from \(D(A^*)\), the domain of \(A^*\), to the whole space \(H_\beta\).
The random measure \(\nu^X(\dd t, G) \equiv \int_E \1_G (\gamma(X_{t-}, y)) F(\dd y) \dd t\) is the \(P\)-compensator of the random measure of jumps \(\mu^X\) associated with \(X\).
To see this, note that the Borel \(\sigma\)-field \(\mathscr{B}(H_\beta)\) is generated by the \(\pi\)-system of cylindrical sets \(\{x \in H_\beta\colon (\la x, y_1\ra_{H_\beta}, ..., \la x, y_n\ra_{H_\beta}) \in G\}\) for \(y_1, ..., y_n \in D(A^*)\) and \(G \in \mathscr{B}(\mathbb{R}^n)\). Here, one uses that \(D(A^*)\) is dense in \(H_\beta\).
Now, the process
\begin{align*}
\dd Z_t &= Z_t \left(\zeta(X_t) \dd X^c_t + \int_E (Y(X_{t-}, y) - 1) \left(\mu^X(\dd t, \dd y) - \nu^X(\dd t, \dd y)\right)\right),\\ Z_0 &= 1,
\end{align*}
defines a local \((\F,P)\)-martingale with the first approach times
\[
\tau_n \equiv \inf(t \in [0, \infty) \colon \|X_t\| \geq n \text{ or } \|X_{t-}\| \geq n)\wedge n \wedge T
\]
as localizing sequence. This can be shown using the local boundedness assumptions on the coefficients. Next, for all \(n \in \mathbb{N}\) one can define a probability measure \(Q_n\) by the Radon-Nikodym derivative \(\dd Q_n = Z_{\tau_n} \dd P\). Using the well-posedness, one obtains that \(Q_n = Q\) on \(\mathscr{F}_{\tau_n}\). We stress that this is no direct consequence of the well-posedness, but relies on the Markovian structure of well-posed MPs. More precisely, one can show that for well-posed MPs, their stopped versions are also well-posed. This is related to the concept of \emph{local uniqueness} as defined in \cite{JS}.
From this identity, the fact that \(P\)-a.s. and \(Q\)-a.s. \(\tau_n \uparrow T\) as \(n \to \infty\) and the optional stopping theorem we deduce that 
\(
Q(G) = E[Z_T \1_G]
\)
for all \(G \in \mathscr{F}\).
Finally, since \(P\)-a.s. \(Z_T > 0\) by the assumption that \(Y\) maps into \((0, \infty)\), \(Q \sim P\) and the claim of theorem follows. 
\end{proof}
Next, we give deterministic conditions, which imply the prerequisites of the previous theorem.
\begin{condition}\label{cond: main}
	\begin{enumerate}
		\item[\textup{(i)}]
		Condition \ref{cond: DC} holds with \(\gamma'\) and \(F'\).
		\item[\textup{(ii)}]
		There exists a Borel function \(\phi \colon E \to [0, \infty)\) such that 
		\[
		|\Gamma(h, y)(z)| \equiv \left|\int_0^z \gamma'(h, y)(x)\dd x\right| \leq \phi(y)
		\]
		for all \(h \in H_\beta, y \in E\) and \(z \in [0, \infty)\).
		Moreover, there is a constant \(L \in (0, \infty)\) such that 
		\[\|a(h) - a(h)\|_{L^2} + \left(\int e^{\phi(z)} \|\gamma'(h, z) - \gamma'(k, z)\|^2_{H_{\beta'}} F'(\dd z)\right)^{\frac{1}{2}} \leq L \|h - k\|_{H_\beta}
		\]
		and
		\begin{align*}
		\|a(h)\|_{L^2} +
		\int e^{\phi(z)} \left(\|\gamma'(h, z)\|^2_{H_{\beta'}} \vee \|\gamma'(h, z)\|^4_{H_{\beta'}}\right) F'(\dd z) \leq L
		\end{align*}
		for all \(h, k \in H_\beta\).
		Moreover, for all \(h \in H_\beta\) the map
		\[
		[0, \infty) \ni z \mapsto \alpha(h) (z) \equiv - \int \gamma'(h, x) \left( e^{\Gamma(h, x)(z)} - 1 \right) F'(\dd x)
		\]
		is absolutely continuous with weak derivative 
		\[
		D \alpha (h) =  \int \gamma'(h, x)^2 e^{\Gamma(h, x)(z)} F'(\dd x) - \int \left(D \gamma'(h, x)\right) \left( e^{\Gamma (h, x)(z)} - 1 \right) F'(\dd x).
		\]
	\end{enumerate}
\end{condition}
Part (ii) implies that a structure preserving EMM exists. This is proven in \cite{FTT1}. 
Let us also comment on the necessity of the conditions. As discussed in \cite[p. 537]{FTT1} the boundedness assumptions in part (ii) cannot be weakened substantially. To see this, note that risk-neutral dynamics driven by a one-dimensional Brownian motion already explode for linear volatility \(a (h) = \textup{const. } h\), see \cite{morton1989arbitrage}.

.

\begin{theorem}\label{theo: A2}
	If Condition \ref{cond: main} holds, then the MP \((F', \xi, a, \gamma', h_0)\) has a unique solution \(Q\). Moreover, if the MP \((F, b, a, \gamma, h_0)\) has a solution \(P\), then \((P, Q)\) is a HJMM pair.
\end{theorem}
\begin{proof}
The existence of \(Q\) is due to \cite[Theorem 3.2]{FTT1} and \cite[Theorem 3.12]{criens17b}. In particular, by \cite[Theorem 3.2]{FTT1}, the discounted bond prices \eqref{eq: dbp} are local \(Q\)-martingales for all maturities. Hence, recalling Theorem \ref{theo: A1}, it suffices to prove that the MP \((F', \xi, a, \gamma')\) is well-posed.
	We use a Yamada-Watanabe-type argument. To bound the technical level of the proof we only sketch the argument.
	Denote by \((\Sigma, \mathscr{A})\) the canonical space for a Brownian motion and a Poisson random measure, see \cite[Display 14.39]{J79}.
	We denote the joint law of the Brownian motion with covariance \(K\) and the Poisson random measure with intensity measure \(\dd t \otimes F'(\dd x)\) by \(P^o\). It is well-known that \(P^o\) is unique, see, e.g., \cite{JS, stroock2010}.
	Let \(P^*_1\) and \(P^*_2\) be two solutions to the MP \((F', \xi, a, \gamma', h_0)\). By \cite[Theorem 3.12]{criens17b}, for \(i = 1, 2\) there exit a filtered probability space which supports a cylindrical Brownian motion \(W^i\), a Poisson random measure \(\mu^i\) with intensity measure \(\dd t \otimes F'(\dd x)\) and a \cadlag adapted process \(Y^i\) such that for all \(h\) in the domain of \(A^*\) we have (up to indistinguishability)
	\begin{align*}
	\la Y^i, h\ra_{H_\beta} = \int_0^\cdot& \left(\la Y^i_{s-}, A^* h\ra_{H_\beta} + \la \xi(Y^i_{s-}), h\ra_{H_\beta}\right)\dd s + \int_0^\cdot \la a(Y^i_{s-}) \dd W^i_s,  h\ra_{H_\beta} 
	\\&+ \int_0^\cdot \int \la \gamma' (Y^i_{s-}, x), h \ra_{H_\beta} (\mu^i (\dd s, \dd x) - \dd s \otimes F'(\dd x))
	\end{align*}
	such that the law of \(Y^i\) coincides with \(P^*_i\).
	
	Let \(P_i\) be the joint law of \((Y^i, (W^i, \mu^i))\), seen as a probability measure on the product space \((\Omega \times \Sigma, \mathscr{F} \otimes \mathscr{A})\). Clearly, the first marginal of \(P_i\) coincides with \(P^*_i\).
	It is well-known that there exists a decomposition
	\[
	P^i (\dd \omega, \dd \sigma) = \mathcal{P}^i (\dd \omega, \sigma) P^o(\dd \sigma), 
	\]
	where \(\mathcal{P}^i\) is a transition kernel from \((\Sigma, \mathscr{A})\) to \((\Omega \times \Sigma, \mathscr{F} \otimes \mathscr{A})\).
	Now, set 
	\[
	P^\star (\dd \omega^1, \dd \omega^2, \dd \sigma) \equiv \mathcal{P}(\dd \omega^1, \sigma) \mathcal{P}^2(\dd \omega^2,\sigma) P^o(\dd \sigma),
	\]
	seen as a probability measure on \((\Omega \times \Omega \times \Sigma, \mathscr{F} \otimes \mathscr{F} \otimes \mathscr{A})\). 
	We denote the generic element on \(\Omega \times \Omega \times \Sigma\) by \((X^1, X^2, (W, \mu))\).
	It follows along the lines of the proof of \cite[Lemma 14.86]{J79} that, up to a \(P^o\)-null set, the random variable \(\mathcal{P}^i(\cdot, G)\) is \(\mathscr{F}_t\)-measurable whenever \(G \in \mathscr{F}_t\). Hence, \cite[Hypothesis 10.43]{J79} is satisfied and we may deduce from \cite[Propositions 10.46 and 10.47]{J79} and \cite[Lemma 4.1]{criens17b} that \(W\) is a Brownian motion with covariance \(K\), \(\mu\) is a Poisson random measure with intensity measure \(\dd t \otimes F'(\dd x)\) and for all \(h\) in the domain of \(A^*\) (up to \(P^\star\)-indistinguishability) 
	\begin{align*}
	\la X^i, h\ra_{H_\beta} = \int_0^\cdot& \left(\la X^i_{s-}, A^* h\ra_{H_\beta} + \la \xi(X^i_{s-}), h\ra_{H_\beta}\right)\dd s + \int_0^\cdot \la a(X^i_{s-}) \dd W_s, h\ra_{H_\beta} 
	\\&+ \int_0^\cdot \int \la \gamma' (X^i_{s-}, x), h \ra_{H_\beta} (\mu (\dd s, \dd x) - \dd s \otimes F'( \dd x)).
	\end{align*}
	In other words, \(X^1\) and \(X^2\) are two (analytically) weak solution processes to 
	\[
	\dd Y_t = \left(AY_t + \xi(Y_t)\right)\dd t + a(Y_t) \dd W_t + \int_E \gamma'(Y_{t-}, y) \left(\mu(\dd t, \dd y) - \dd t \otimes F'( \dd y)\right)
	\]
	w.r.t. the same driving noise \(W\) and \(\mu\).
	By \cite[Corollary 10.9]{FTT2} these solutions are indistinguishable, i.e. \(P^\star (X^1_t = X^2_t \text{ for all } t \in [0, T]) = 1\). Hence, it follows readily from the definition of \(P^\star\) that \(P^*_1 = P^*_2\).
	Thus, we conclude the proof.
\end{proof}

Next, we also give explicit conditions for the existence of \(P\).
\begin{condition}\label{cond: main2}
	It holds that \(\int \|\gamma(0, z)\|^2_{H_{\beta'}} F(\dd z) < \infty\).
	Moreover, there exists a constant \(L \in (0, \infty)\) such that 
	\begin{align*}
	\|b(h) - b(k))\|_{H_\beta} + \|a(h) - a(k)\|_{L^2} &\leq L \|h - k\|_{H_\beta},\\
	\left(\int  \|\gamma(h, z) - \gamma(k, z)\|^2_{H_{\beta'}} F(\dd z)\right)^{\frac{1}{2}}&\leq L \|h - k\|_{H_\beta},
	\end{align*}
	for all \(h, k \in H_\beta\).
\end{condition}
The Lipschitz conditions in part (ii) are classical existence and uniqueness conditions. In this degree of generality we are not aware of better conditions ensuring existence.
\begin{corollary}
	Suppose that the Conditions \ref{cond: main} and \ref{cond: main2} hold. Then there are unique solutions \(P\) and \(Q\) to the MPs \((F, b, a, \gamma, h_0)\) and \((F', \xi, a, \gamma', h_0)\).
	Moreover, \((P, Q)\) is a HJMM pair.
\end{corollary}
\begin{proof}
	The existence and uniqueness of \(P\) is due to \cite[Corollary 10.9]{FTT2}, \cite[Theorem 2.1]{FTT1} and \cite[Theorem 3.12]{criens17b} together with a Yamada-Watanabe-type argument as sketched in the proof of Theorem \ref{theo: A1}. Now, Theorem \ref{theo: A1} yields the assertion. 
\end{proof}
Next, we give condition that the real-world and risk-neutral dynamics only produce non-negative forward curves.

\subsection{Positivity Preserving HJMM Pairs}\label{sec: 2.3}
From an applications point of view, we would like to choose HJMM Pairs which only produce non-negative forward curves.
In this section we transfer some conditions from \cite{FTT1} to our methodology. Denote \(\mathfrak{P} \equiv \{h \in H_\beta \colon h (s) \geq 0 \text{ for all } s \in [0, \infty)\}\).
\begin{definition}
	We call a HJMM pair \((P, Q)\) to be \emph{positivity preserving} if \(P(X_t \in \mathfrak{P} \text{ for all } t \in [0, T]) = Q(X_t \in \mathfrak{P} \text{ for all } t \in [0, T]) = 1.\)
\end{definition}
\begin{remark}\label{rem: PP}
	Since \(Q \sim P\), a HJMM pair \((P, Q)\) is positivity preserving if, and only if, either \(P(X_t \in \mathfrak{P} \text{ for all } t \in [0, T]) = 1\) or \(Q(X_t \in \mathfrak{P} \text{ for all } t \in [0, T]) = 1\).
\end{remark}
In view of the previous remark, the conditions of \cite{FTT1} for positivity preserving risk-neutral HJMM models directly transfer to our setting. 
For completeness, let us restate them.
\begin{condition}\label{cond: PP1}
The volatility coefficient \(a\) is an element of \(C^2(H_\beta, L_2)\) and the vector field \(h \mapsto \la D a, a \ra_{L^2}\) is Lipschitz continuous.
\end{condition}	
\begin{condition}\label{cond: PP2}	
	Let \((e_i)_{i \in \mathbb{N}}\) be an orthonormal basis of \(U\) and \(\gamma'\) and \(F'\) as in Condition \ref{cond: DC}.
	\begin{enumerate}
		\item[\textup{(i)}]
		For all \(h \in \mathfrak{P}\) and \(F'\)-a.a. \(x \in E\) it holds that \(h + \gamma'(h, x) \in \mathfrak{P}\).
		\item[\textup{(ii)}]
				For all \(i \in \mathbb{N}, s \in (0, \infty), h \in \{k \in H_\beta \colon k(s) = 0\}\) and \(F'\)-a.a. \(x \in E\) it holds that \(\gamma'(h, x)(s) = a (h) e_i (s) = 0\).
	\end{enumerate}
\end{condition}
\begin{theorem}
	If the Conditions \ref{cond: main} and \ref{cond: PP1} hold, then there exists a (unique) solution \(Q\) to the MP \((F', \xi, a, \gamma', h_0)\) and for any solution \(P\) of the MP \((F, b, a, \gamma, h_0)\) the tuple \((P, Q)\) is a HJMM pair which is positivity preserving if, and only if, Condition \ref{cond: PP2} holds.
\end{theorem}
\begin{proof}
	This follows from Theorem \ref{theo: A2}, Remark \ref{rem: PP} and \cite[Theorem 4.13]{FTT1}.
\end{proof}

In summary, we gave deterministic conditions for the existence of a pair of real-world and risk-neutral dynamics of a HJMM framework driven by Hilbert space valued Brownian motion and a Poisson random measure. In particular, our conditions imply that the real-world measure and the risk-neutral measure both solve MPs. Based on results from \cite{FTT1}, we gave conditions for a HJMM pair which produces positive forward curves under both the risk-neutral and the real-world measure.
\bibliographystyle
{chicago}
\bibliography{References}

\end{document}